\documentclass[12pt]{amsart}
\usepackage{amsmath,amssymb,graphicx,amsthm,bm,hyperref,mathrsfs,enumitem}
\usepackage[utf8]{inputenc}
\usepackage[pagewise]{lineno}

\newtheorem{theorem}{Theorem}[section]
\newtheorem{lem}[theorem]{Lemma}
\newtheorem{cor}[theorem]{Corollary}
\newtheorem{prop}[theorem]{Proposition}
\newtheorem{definition}[theorem]{Definition}
\newtheorem{example}[theorem]{Example}
\newtheorem{remark}[theorem]{Remark}
\numberwithin{equation}{section}

\begin{document}

\title[Embedding of Sobolev spaces into BMO]{On some refinements of the embedding of critical Sobolev spaces into BMO}
\author[A. Butaev]{Almaz Butaev}

\begin{abstract}

We introduce the non-homogeneous analogs of function spaces studied by Van Schaftingen’s. We show that these classes refine the embedding $W^{1,n}(\mathbb{R}^n)\subset bmo(\mathbb{R}^n)$. The analogous results established on bounded Lipschitz domains and Riemannian manifolds with bounded geometry.

\end{abstract}

\maketitle

\section{Introduction}
Let $f$ be a locally integrable function on $\mathbb{R}^n$. Given a cube $Q\subset \mathbb{R}^n$ (henceforth by a cube we will understand a cube with sides parallel to the axes), we denote the average of $f$ over $Q$ by $f_Q$, i.e.
\[
f_Q = \frac{1}{|Q|} \int_Q f(x) dx,
\]
where $|Q|$ is the Lebesgue measure of $Q$.

In 1961 John and Nirenberg introduced the space of functions of bounded mean oscillation ($\rm{BMO})$.
\begin{definition}
We say that $f\in \rm{BMO}(\mathbb{R}^n)$ if
\[
\|f\|_{\rm{BMO}}:= \sup_{Q} \frac{1}{|Q|}\int_Q |f(x)-f_Q| dx < \infty.
\]
Note that $\|\cdot\|_{\rm{BMO}}$ is a norm on the quotient space of functions modulo constants.
\end{definition}

Functions of bounded mean oscillations turned out to be the right substitute for $L^\infty$ functions in a number of questions in analysis. In particular, the embedding theorem of Gagliardo-Nirenberg-Sobolev (see e.g.\ \cite{stein2016singular}, Chapter V) asserts that for any $p\in[1,n)$ there exists $C_p$ such that
\[
\|f\|_{L^{np/(n-p)}} \leq C_p \|\nabla f\|_{L^p}, \forall f\in \mathcal{D}.
\]
The inequality fails for $p=n$, so we do not have the embedding $W^{1,n}$ into $L^\infty$. However, it follows from the Poincare inequality that for some constant $C>0$,
\[
\|f\|_{\rm{BMO}} \leq C \|\nabla f\|_{L^n}, \forall f\in \mathcal{D}
\]
and therefore $\mathring{W}^{1,n}$ is continuously embedded into $\rm{BMO}(\mathbb{R}^n)$.

Based on one inequality established by Bourgain and Brezis in \cite{BB_New_estimates_Laplacian_Hodge_systems}, Van Schaftingen \cite{Schaftingen_function_spaces_BMO_Sobolev} defined a scale of spaces $D_k$ using the $k$-differential forms
\[
\Phi(x) = \sum_{1\leq i_1<\dots<i_k\leq n} \phi_{i_1,\dots i_k}(x) dx^{i_1}\wedge \dots \wedge dx^{i_k}
\]
as follows
\begin{definition}
For $1\leq k \leq n$, $D_k$ is defined as
\[
D_{k}(\mathbb{R}^n) = \{u \in \mathcal{D}'(\mathbb{R}^n): \|u\|_{D_{k}}<\infty \},
\]
where
\[
\|u\|_{D_{k}} := \sup\{|u(\phi_{i_1,\dots,i_k})|: \Phi \in  \mathcal{D}(\mathbb{R}^n;\Lambda^k(\mathbb{R}^n)),\  d\Phi=0, \ \|\Phi\|_{L^1} \leq 1\}.
\]
\end{definition}
It was shown in \cite{Schaftingen_function_spaces_BMO_Sobolev} that the $D_k$ classes lie strictly between the critical Sobolev spaces and $\rm{BMO}(\mathbb{R}^n)$, refining the classical embedding $\mathring{W}^{1,n} \subset \rm{BMO}$. More precisely, the following proper inclusions are continuous
\[
\mathring{W}^{1,n} \subset D_{n-1} \subset \dots \subset D_1 \subset \rm{BMO}.
\]

From the point of view of some applications to PDEs, as function spaces $D_k$ ($k<n$) lack certain  ``useful" properties: multiplications by smooth cut-off functions are not necessarily bounded operators on $D_k$ and $D_k$ are not invariant under all smooth changes of variables.

In this paper, we introduce the non-homogeneous analogs of Van Schaftingen's classes $D_k$, which we denote by $d^k(\mathbb{R}^n)$.
\begin{definition}
Let $1\leq k \leq n$. We say that $u\in \mathcal{D}'(\mathbb{R}^n)$ belongs to $d^k(\mathbb{R}^n)$ if
\begin{equation}
\label{d_k_def}
 \sup_{\|\Phi\|_{\Upsilon^1_k(\mathbb{R}^n)}\leq 1} \max_{I} |u(\phi_I)| <\infty,
\end{equation}
where the supremum is taken over all $k$-differential forms $\Phi = \sum_{I} \phi_I dx^I$, $\phi_I \in \mathcal{D}(\mathbb{R}^n)$ and $\|\Phi\|_{\Upsilon^1_k} = \|\Phi\|_{L^1} + \|d \Phi\|_{L^1}$. We will denote this supremum by $\|u\|_{d^k}$.
\end{definition}

It is useful to compare the defined classes $d^k(\mathbb{R}^n)$ with $D_k(\mathbb{R}^n)$. First of all, $d^k(\mathbb{R}^n) \subset D_k(\mathbb{R}^n), \ k=1,2,\dots n$ as sets. As Banach spaces $D_k(\mathbb{R}^n)$ are classes of functions modulo constants, while in $d^k(\mathbb{R}^n)$ two functions that differ by a non-zero constant are considered as different elements.

In contrast to $D_k$ spaces, the smooth change of variables and multiplications by cut-off functions are invariant operations on $d^k$. In particular, this allows to define $d^k$ on certain Riemannian manifolds.
In Section 2, we recall some facts from the theory of local Hardy spaces, which will be used later.
In Section 3, we prove the following theorem
\begin{theorem}
\label{main_thm}
$d^1(\mathbb{R}^n)$ is continuously embedded into the space $bmo(\mathbb{R}^n)$ and $\exists C>0$ so that for any $u\in d^k(\mathbb{R}^n)$, $1\leq k\leq n$
\[
\|u\|_{bmo} \leq C \|u\|_{d^k}.
\]
\end{theorem}
Combining this theorem with the result of Van Schaftingen \cite{Schaftingen2004181}, it shows that the $d^k$ classes refine the embedding $W^{1,n}(\mathbb{R}^n)\subset bmo(\mathbb{R}^n)$, where $bmo$ is the local BMO space of Goldberg \cite{Goldberg} in the sense that
\[
W^{1,n}\subset d^{n-1} \subset \dots \subset d^1 \subset bmo.
\]
We also prove that continuous $d^{n-1}$ functions can be characterized in terms of line integrals, similarly to the inequality of  Bourgain, Brezis and Mironescu \cite{bourgain2004h}
\begin{theorem}
Let $u\in \mathcal{D}(\mathbb{R}^n)$. Then $u\in d^{n-1}(\mathbb{R}^n)$ if and only if
\[
\sup\limits_{\partial \gamma=\emptyset}\frac{1}{|\gamma|} \left| \int_\gamma u(t) \tau(t) dt \right| + \sup\limits_{|\gamma|\geq 1}\frac{1}{|\gamma|} \left| \int_\gamma u(t) \tau(t) dt \right|<\infty,
\]
where the suprema are taken over smooth curves $\gamma$ with finite lengths $|\gamma|$, boundaries $\partial \gamma$ and unit tangent vectors $\tau$.
\end{theorem}

As an application of $d^k$ classes for PDEs, the following fact
is established
\begin{theorem}
Let $n\geq 2$, $F\in L^1(\mathbb{R}^n;\mathbb{R}^n)$ and $div F\in L^1(\mathbb{R}^n)$. Then the system $(I-\Delta) U = F$ admits a unique solution $U$ such that
\begin{itemize}
    \item If $n=2$, then
    \[
    \|U\|_{\infty} + \|\nabla U\|_{2} \leq C (\|F\|_1 + \|div F\|_1)
    \]
    \item If $n\geq 3$, then
    \[
    \|U\|_{n/(n-2)} + \|\nabla U\|_{n/n-1} \leq C (\|F\|_1 + \|div F\|_1)
    \]
\end{itemize}

\end{theorem}

In Section 4, we introduce the localized versions of $d^k$ spaces on bounded Lipschitz domains $\Omega$.
The main result of Section 4 is the proof of the following fact, which was conjectured by Van Schaftingen \cite{Schaftingen_function_spaces_BMO_Sobolev} for the bmo spaces on domains (see Definition \ref{def_bmo_r} below).

\begin{theorem}
Any $u\in d^1(\Omega)$ is a $bmo_r(\Omega)$ function as there exists $C>0$ such that
\[
\|u\|_{bmo_r(\Omega)} \leq C \|u\|_{d^1(\Omega)} \  \forall u\in d^1(\Omega).
\]
\end{theorem}

In Section 5, we define $d^k$ classes on Riemannian manifolds with bounded geometry and based on the results of Section 3 we prove the refined embeddings between critical Sobolev space and $bmo$ on such manifolds.
\begin{theorem}
Let $M$ be the Riemannian manifold with bounded geometry. Then the following continuous embeddings are true
\[
W^{1,n}(M) \subset d^{n-1}(M) \subset \dots \subset d^1(M) \subset bmo(M).
\]
\end{theorem}

\section{Preliminaries}
Let $\Omega\subset \mathbb{R}^n$ be open. We will use the Schwartz notations: $\mathcal{E}(\Omega)$ will denote the class of smooth functions on $\Omega$,  $\mathcal{D}(\Omega)$ and $\mathcal{S}(\Omega)$ will stand for compactly supported smooth functions and smooth functions rapidly decaying at infinity with all their derivatives. By $\mathcal{D}^k(\Omega)$ we denote the class of $k$-differential forms with $\mathcal{D}(\Omega)$ components.
All $L^p$ spaces in this paper are considered relative to the Lebesgue measure. For the differential form of order $k$, $\Phi = \sum_{|I|=k} \phi_I dx^I$, we will use the notation
\[
\|\Phi\|_{L^1_k} = \sum\limits_{I} \|\phi_I\|_{L^1}.
\]
However, often when it does not create confusion we will omit the subscript $k$ and simply write $\|\Phi\|_{L^1}$ or $\|\Phi\|_1$.
\subsection{Local Hardy and BMO spaces of Goldberg}
We recall the definition and basic properties of the local Hardy space $h^1(\mathbb{R}^n)$ introduced by Goldberg \cite{Goldberg}.

Let us fix $\phi\in \mathcal{S}(\mathbb{R}^n)$ such that $\int \phi \neq 0$. For $f\in L^1(\mathbb{R}^n)$, we define the local maximal function $m_\phi f(x) $ by
\[
m_\phi f(x) = \sup\limits_{0<t<1} |\phi_t*f(x)|,
\]
where $\phi_t(y) = t^{-n} \phi (\frac{y}{t})$.

\begin{definition}
We say that $f$ belongs to the local Hardy space $h^1(\mathbb{R}^n)$ if $m_\phi f \in L^1(\mathbb{R}^n)$ and we put
\[
\|f\|_{h^1}:=\|m_\phi f\|_{L^1}.
\]
\end{definition}
It is useful to compare $h^1$ with the classic real Hardy space $H^1(\mathbb{R}^n)$, which can be defined using the global maximal function $M_\phi$,
\[
M_\phi f(x):= \sup_{t>0} |\phi_t*f(x)|, \ f\in L^1(\mathbb{R}^n).
\]
\begin{definition}
We say that $f$ belongs to the Hardy space $H^1(\mathbb{R}^n)$ if $M_\phi f \in L^1(\mathbb{R}^n)$, and we put
\[
\|f\|_{H^1}:=\|M_\phi f\|_{L^1}.
\]
\end{definition}
It follows from the definitions of the maximal functions that $m_\phi f(x) \leq M_\phi f(x)$ for any $f\in L^1$ and $x\in \mathbb{R}^n$. Therefore $H^1 \subset h^1$. One of the reasons why it is often more convenient to deal with a larger space $h^1$ instead of $H^1$ is that $\mathcal{S}(\mathbb{R}^n) \subset h^1(\mathbb{R}^n)$, while any $f\in H^1(\mathbb{R}^n)$ has to satisfy $\int_{\mathbb{R}^n} f = 0$. Moreover, the following result is true.
\begin{lem}[\cite{Goldberg}]
\label{lem_h1_density}
The space $\mathcal{D}(\mathbb{R}^n)$ is dense in $h^1(\mathbb{R}^n)$.
\end{lem}

It is important to note that $f\in h^1(\mathbb{R}^n)$ and $\int\limits_{\mathbb{R}^n} f = 0$ do not imply that $f\in H^1(\mathbb{R}^n)$ (see Theorem 3 in \cite{Goldberg}). However, the following is true
\begin{lem}
\label{lem_H1_h1}
If $f\in h^1(\mathbb{R}^n)$, $\int\limits_{\mathbb{R}^n} f(x) dx =0$ and $supp \ f\subset B$, where $B$ is a bounded subset of $\mathbb{R}^n$, then there exists $C_B>0$ such that
\[
\|f\|_{H^1} \leq C_B \|f\|_{h^1}.
\]
\end{lem}

\begin{definition}[\cite{Goldberg}]
We say that $f\in L^1_{loc}(\mathbb{R}^n)$ belongs to $bmo(\mathbb{R}^n)$ if
\[
\|f\|_{bmo}:= \sup_{l(Q)\leq 1} \frac{1}{|Q|} \int_Q |f(x) - f_Q| dx  + \sup_{l(Q)\geq 1} \frac{1}{|Q|} \int_Q |f(x)| dx < \infty,
\]
where $f_Q = \frac{1}{|Q|} \int_Q f(y) dy$ and $Q$ are cubes with sides parallel to the axes, of side-length $l(Q)$.
\end{definition}

It is clear that $bmo(\mathbb{R}^n)$ is a subspace of $\rm{BMO}(\mathbb{R}^n)$. Moreover, if $\|f\|_{bmo} = 0$, then $f=0$ a.e. on $\mathbb{R}^n$, unlike in $\rm{BMO}(\mathbb{R}^n)$, where constant functions are identified with $f\equiv 0$.

The following theorem of Goldberg shows the relation between $h^1$ and $bmo$ and the boundedness of pseudo-differential operators of degree zero on $h^1$.
\begin{theorem} [\cite{Goldberg}]
The space $bmo(\mathbb{R}^n)$ is isomorphic to the space of continuous linear functionals on $h^1(\mathbb{R}^n)$.
\end{theorem}

\begin{theorem} [{\cite{Goldberg}}]
\label{thm_goldberg_OPS}
If $T\in \rm{OPS}^0$, then there exists a constant $C>0$ such that
\[
\|T f\|_{h^1} \leq C \|f\|_{h^1} \text{ for any } f\in \mathcal{S}(\mathbb{R}^n).
\]
Therefore, any $T\in \rm{OPS^0}$ can be extended to a continuous linear operator on $h^1(\mathbb{R}^n)$.
\end{theorem}

\subsection{Local Hardy and BMO spaces on Lipschitz domains}
The BMO and Hardy spaces on bounded Lipschitz domains were studied in \cite{CDS}, \cite{CKS} and \cite{Myachi_Hp} (see also \cite{Jones_Extension_theorems_for_BMO} and \cite{Strichartz_Hardy_on_manifolds}).

\begin{definition} \cite{CKS}, \cite{Myachi_Hp}

Let $\Omega\subset \mathbb{R}^n$ be a bounded Lipschitz domain. The space $h^1_r(\Omega)$ consists of elements of $L^1(\Omega)$ which are the restrictions to $\Omega$ of elements of $h^1(\mathbb{R}^n)$, i.e.
\[
h^1_r(\Omega) = \{f\in L^1(\Omega): \exists F\in h^1(\mathbb{R}^n): F=f \text{ on } \Omega\}.
\]
We can consider this as a quotient space equipped with the quotient norm
\[
\|f\|_{h^1_r(\Omega)} := \inf\{\|F\|_{h^1(\mathbb{R}^n)}: F=f \text{ on } \Omega \}.
\]
\end{definition}
\begin{definition} \cite{CDS}
The space $h^1_z(\Omega)$ is defined to be the subspace of $h^1(\mathbb{R}^n)$ consisting of those elements which are supported on $\overline{\Omega}$.
\end{definition}

Like in the case of $\mathbb{R}^n$, smooth and compactly supported functions are dense in these spaces:

\begin{lem} \cite{caetano2000approximation}
Let $\Omega$ be a bounded Lipschitz domain. Then the space $\mathcal{D}(\Omega)$ is dense in $h^1_r(\Omega)$.
\end{lem}

\begin{lem}
\label{lem_density_h1_domains}
Let $\Omega$ be a domain of $\mathbb{R}^n$. Then the set of $\mathcal{D}(\Omega)$ functions is dense in $h^1_z(\Omega)$.
\end{lem}

The BMO analogs on $\Omega$ are defined as follows.
\begin{definition}
The space $bmo_z(\Omega)$ is defined to be a subspace of $bmo(\mathbb{R}^n)$ consisting of those elements which are supported on $\bar{\Omega}$, i.e.
\[
bmo_z(\Omega) = \{g\in bmo(\mathbb{R}^n): g=0 \text{ on } \mathbb{R}^n\setminus \bar{\Omega} \}
\]
with
\[
\|g\|_{bmo_z(\Omega)}= \|g\|_{bmo(\mathbb{R}^n)}.
\]
\end{definition}
\begin{definition} \cite{CDS}
\label{def_bmo_r}
Let $\Omega$ be a bounded Lipschitz domain. A function $g\in L^1_{loc}(\Omega)$ is said to belong to $bmo_r(\Omega)$ if
\[
\|g\|_{bmo_r(\Omega)} = \sup\limits_{|Q|\leq 1} \frac{1}{|Q|} \int_Q |g(x)-g_Q| dx + \sup\limits_{|Q|>1} \frac{1}{|Q|} \int_Q |g(x)| dx < \infty,
\]
where suprema are taken over all cubes $Q\subset \Omega$.
The space of such functions  equipped with norm $\|\cdot\|_{bmo_r(\Omega)}$ is called $bmo_r(\Omega)$.
\end{definition}

\begin{theorem} [\cite{Chang}, \cite{Myachi_Hp}]
\label{thm_bmoz_as_dual}
The space $bmo_z(\Omega)$ is isomorphic to the dual of $h^1_r(\Omega)$.
\end{theorem}

\begin{theorem} [\cite{Chang}, \cite{jonsson1984hardy}]
\label{thm_bmor_as_dual}
The space $bmo_r(\Omega)$ is isomorphic to the dual of $h^1_z(\Omega)$.
\end{theorem}

\subsection{$H^1_z(\Omega)$ space}
We will also need the following function space:
\begin{definition} \cite{CKS}
The space $H^1_z(\Omega)$ is defined to be the subspace of $H^1(\mathbb{R}^n)$ consisting of those elements which are supported on $\overline{\Omega}$.
\end{definition}

One of the alternative ways to define $H^1_z(\Omega)$ is to evoke the notion of atoms.
\begin{definition}
An $H^1_z(\Omega)$ atom is a Lebesgue measurable function $a$ on $\mathbb{R}^n$, supported on a cube $Q\subset \Omega$, such that
\[
\|a\|_{L^2(Q)} \leq |Q|^{-1/2}
\]
and
\[
\int_{Q} a(x) dx = 0.
\]
\end{definition}
Any $H^1_z(\Omega)$ function can be represented as a series of $H^1_z$ atoms in the following sense:
\begin{theorem} [Theorem 3.3 in \cite{CKS}]
\label{thm_CKS_H1_atoms}
Let $\Omega$ be a bounded Lipschitz domain and $f\in L^1(\Omega)$. Then $f\in H^1_z(\Omega)$ if and only if there exist a sequence of $H^1_z(\Omega)$ atoms $\{a_k\}$ and real numbers $\{\lambda_k\} \subset \mathbb{R}$ such that $\sum |\lambda_k| < \infty$ and
\[
\sum_{k} \lambda_k a_k \to f \text{ in } \mathcal{D}'(\Omega).
\]
Furthermore,
\[
\|f\|_{H^1} \approx  \inf\{\sum_k |\lambda_k|: f= \sum_k \lambda_k a_k\},
\]
where the infimum is taken over all atomic decompositions of $f$.
\end{theorem}

\section{$d^k$ spaces on $\mathbb{R}^n$}

\begin{definition}
Let $1\leq k \leq n$. We say that $u\in \mathcal{D}'(\mathbb{R}^n)$ belongs to $d^k(\mathbb{R}^n)$ if
\begin{equation}
\label{d_k_def}
 \sup_{\|\Phi\|_{\Upsilon^1_k(\mathbb{R}^n)}\leq 1} \max_{|I|=k} |u(\phi_I)| <\infty,
\end{equation}
where the supremum is taken over all $\Phi = \sum_{|I|=k} \phi_I dx^I \in \mathcal{D}^k(\mathbb{R}^n)$ and $\|\Phi\|_{\Upsilon^1_k} = \|\Phi\|_{L^1_k} + \|d \Phi\|_{L^1_{k+1}}$. We will denote this supremum by $\|u\|_{d^k}$.
\end{definition}

\begin{remark}
\label{rem_d_space_upsilon}
It is not difficult to show that the class of compactly supported $\Upsilon^1_k(\mathbb{R}^n)$ forms is dense in $\Upsilon^1_k(\mathbb{R}^n)$. This suggests that the domain of $u\in d^k(\mathbb{R}^n)$ can be extended to include components of all $\Upsilon^1_k(\mathbb{R}^n)$ forms. Let $u\in \mathcal{D}'(\Omega)$ and  $\tilde{u}$ be a linear map from $\mathcal{D}^k(\Omega)$ to $\left(\mathbb{R}^{\binom{n}{k}}, \|\cdot\|_{max}\right)$, associated to $u$ by
\[
\tilde{u}\left(\sum\limits_{|I|=k} \phi_I dx^I\right) = (u(\phi_I)).
\]
Then $u\in \mathcal{D}'(\mathbb{R}^n)$ belongs to $d^k(\mathbb{R}^n)$, if and only if $\tilde{u}$ can be extended to a bounded linear map from $\Upsilon^1_k(\mathbb{R}^n)$ to $\binom{n}{k}$ dimensional Euclidean space equipped with the $max$ norm.
\end{remark}

Note that $\Upsilon^1_n(\mathbb{R}^n) = L^1(\mathbb{R}^n)$, so $d^n(\mathbb{R}^n)$ is isomorphic to $L^\infty(\mathbb{R}^n)$.
\begin{lem}
\label{lem_emb}
Let $1\leq k<l \leq n$ and $u\in d^l(\mathbb{R}^n)$. Then $u \in d^{k}(\mathbb{R}^n)$ and $\|u\|_{d^{k}(\mathbb{R}^n)} \leq \|u\|_{d^{l}(\mathbb{R}^n)}$.
In other words, the following embeddings are continuous
\[
d^n (\mathbb{R}^n) \subset d^{n-1} (\mathbb{R}^n) \subset \cdots \subset d^{1} (\mathbb{R}^n)  \]
\end{lem}
\begin{proof}
It is enough to consider the case $k=l-1$, because the general case will follow from it by induction. Let $1\leq l\leq n$, $u\in d^l(\mathbb{R}^n)$ and
\[
\Phi(x) = \sum_{|I| = l-1} \phi_{I}(x) dx^{I} \in \mathcal{D}^{l-1}(\mathbb{R}^n).
\]
We need to show for any component $\phi_I$,
\[
|u(\phi_I)|\leq \|u\|_{d^l} \|\Phi\|_{\Upsilon^1_{l-1}}.
\]
Fix any such $I$. Since $|I|=l-1<n$, there exists $j\in [1,n]$ such that $dx^I \wedge dx^j\neq 0$. Put $
 \tilde{\Phi}(x) = \Phi(x) \wedge dx^{j}$.
Then $\tilde{\Phi} \in \mathcal{D}^{l}$ and $\|\tilde{\Phi}\|_{\Upsilon^1_{l}} \leq \|\Phi\|_{\Upsilon^1_{l-1}}$. Moreover, by construction, one of the components of $\tilde{\Phi}$ equals to $\pm \phi_I dx^I\wedge dx^j$. Since $u\in d^l(\mathbb{R}^n)$, we have
\[
|u(\phi_I)| \leq \|u\|_{d^l} \|\tilde{\Phi}\|_{\Upsilon^1_{l}} \leq \|u\|_{d^l} \|\Phi\|_{\Upsilon^1_{l-1}}.
\]
\end{proof}

The following theorem follows immediately from the definition of $d^k$ spaces and the result of Van Schaftingen \cite{Schaftingen2004181}.
\begin{theorem}
\label{thm_VS_my_version}
$W^{1,n}(\mathbb{R}^n)$ is continuously embedded into $d^{n-1}(\mathbb{R}^n)$ as $\exists C>0$ so that for any $u\in W^{1,n}$
\[
\|u\|_{d^{n-1}} \leq C \|u\|_{W^{1,n}}.
\]
\end{theorem}

One of main results in this section is the following
\begin{theorem}
\label{main_thm}
$d^1(\mathbb{R}^n)$ is continuously embedded into the space $bmo(\mathbb{R}^n)$ and $\exists C>0$ so that for any $u\in d^k(\mathbb{R}^n)$, $1\leq k\leq n$
\[
\|u\|_{bmo} \leq C \|u\|_{d^k}.
\]
\end{theorem}
\begin{remark}
This result is a non-homogeneous analogue of the main theorem in \cite{Schaftingen_function_spaces_BMO_Sobolev}. We adapt the proof of that theorem to the non-homogeneous setting.
\end{remark}
\begin{proof}
By Lemma \ref{lem_emb}, it is enough to prove the case $k=1$. The argument is based on the fact that $bmo(\mathbb{R}^n)$ is the dual space of $h^1(\mathbb{R}^n)$. We claim that given $f\in \mathcal{D}(\mathbb{R}^n)$, there exist $n$ differential forms $\{\Phi^j\}_{j=1}^n \subset \Upsilon^1_1(\mathbb{R}^n)$ such that for some $C$ independent of $f$,
\begin{equation}
\label{tempo1}
\|\Phi^j\|_{\Upsilon^1_1} \leq C \|f\|_{h^1},
\end{equation}
\begin{equation}
\label{tempo2}
f=\sum_{i=1}^n \phi^i_i,
\end{equation}
where
\[
\Phi^j = \sum\limits_{i=1}^n \phi^j_i dx^i.
\]
Assuming the claim the proof is easy. Let $u\in d^1(\mathbb{R}^n)$. For arbitrary $f\in \mathcal{D}(\mathbb{R}^n)$, let $\Phi^j$ be such that (\ref{tempo1}) and (\ref{tempo2}) are true. Then by the Remark \ref{rem_d_space_upsilon} we can apply $u$ to $\phi^i_i$ to have
\begin{equation}
\label{temp3}
|u(f)| \leq \sum_{i=1}^n |u(\phi^i_i)|\leq \sum_{i=1}^n \|u\|_{d^1} \|\Phi^i\|_{\Upsilon^1_1} \leq C n \|u\|_{d^1} \|f\|_{h^1}.
\end{equation}
By the density of $\mathcal{D}$ in $h^1$ and the duality $bmo = (h^1)'$,  we conclude that $u\in bmo(\mathbb{R}^n)$.

In order to prove the claim, let $f\in \mathcal{D}$ be arbitrary and consider the equation
\[
(I-\Delta) v = f \text{ in } \mathbb{R}^n.
\]
Then $v=\mathcal{J}(f)$, where $\mathcal{J}$ is a convolution operator whose kernel is the Bessel potential of order 2, $G_2$.
For $j\in [1,n]$, let
\[
\Phi^j  = \sum_{i=1}^n \left( \frac{\mathcal{J} }{n} - \partial_i\partial_j \mathcal{J}\right)(f) dx^i.
\]
Since $f\in \mathcal{D}\subset \mathcal{S}$, all components of $\Phi^j$ are $\mathcal{S}$ functions and
\[
d \Phi^j  = \sum_{1\leq i<k \leq n} \left(\frac{\partial_i \mathcal{J}- \partial_k \mathcal{J}}{n}\right) (f) \ dx^i\wedge dx^k.
\]
It is clear that,
\[
\frac{\mathcal{J}}{n} - \partial_i \partial_j \mathcal{J} \in OPS^{-2}(\mathbb{R}^n) + OPS^0(\mathbb{R}^n) \subset OPS^0(\mathbb{R}^n)
\]
and
\[
\left(\frac{\partial_i \mathcal{J}- \partial_k \mathcal{J}}{n}\right) \in OPS^{-1}(\mathbb{R}^n) \subset OPS^0(\mathbb{R}^n).
\]
Recalling Theorem \ref{thm_goldberg_OPS}, we see that the components of $\Phi^j$ and $d\Phi^j$ are $h^1$ functions and for some $C$ independent of $f$,
\[
\|\Phi^j\|_{L^1_1} + \|d\Phi^j\|_{L^1_2}  \leq C \|f\|_{h^1},
\]
which proves (\ref{tempo1}).
Finally, $\{\Phi^j\}$ satisfy (\ref{tempo2}) for
\[
\sum_{i=1}^n \left( \frac{\mathcal{J} }{n} - \partial_i\partial_i \mathcal{J}\right)f  = \mathcal{J}(f) - \Delta \mathcal{J}(f) = (I-\Delta) \mathcal{J}(f) = f.
\]
\end{proof}
\begin{cor}
For $1\leq k\leq n$, the space $d^k(\mathbb{R}^n)$ equipped with the norm $\|\cdot\|_{d^k}$ is a Banach space.
\end{cor}
\begin{proof}
Let $\{u_m\}_{m=0}^\infty$ be a Cauchy sequence in $d^k$. The above theorem shows that $u_m$ is a Cauchy sequence in $bmo(\mathbb{R}^n)$. Since $bmo$ is a complete Banach space, there exists $u \in bmo(\mathbb{R}^n)$, such that $u_m\to u$ in $\|\cdot\|_{bmo}$. Moreover, for any $\Phi=\sum_{|I|=k} \phi_I dx^I \in \mathcal{D}^k(\mathbb{R}^n)$ and $j\geq 0$, using duality of $bmo$ and $h^1$ and the fact that each $\phi_I\in \mathcal{D}\subset h^1$,
\[
\label{temp_cor}
\left|\int (u_j - u) \phi_I \right| = \lim\limits_{m\to \infty } \left|\int (u_j - u_m) \phi_I  \right| \leq
\]
\[
\leq \lim\limits_{m\to \infty} \|u_j - u_m\|_{d^k} \|\Phi\|_{\Upsilon^1_k},
\]
which shows that  $u\in d^k(\mathbb{R}^n)$, and $\|u_j - u\|_{d^k} \to 0$, as $j\to \infty$.
\end{proof}
Summing up the results of this section, we can now say that for $1\leq k\leq n$,
\[
W^{1,n}(\mathbb{R}^n) \subset d^{n-1}(\mathbb{R}^n) \subset \cdots \subset d^1(\mathbb{R}^n) \subset bmo(\mathbb{R}^n).
\]

\subsection{$v^k$ classes}
\begin{definition}
Let $1\leq k\leq n$. We define the class $v^k(\mathbb{R}^n)$ as the closure of $C_0(\mathbb{R}^n)$ functions in the norm $\|\cdot\|_{d^k}$. Here
\[
C_0(\mathbb{R}^n) = \{u:\in C(\mathbb{R}^n): \lim\limits_{|x|\to \infty} u(x) = 0\}.
\]
\end{definition}
First of all we notice that by Proposition \ref{lem_emb}, $v^k(\mathbb{R}^n)$ form a monotone family of spaces
\[
v^n(\mathbb{R}^n) \subset v^{n-1}(\mathbb{R}^n) \subset \cdots \subset v^1(\mathbb{R}^n).
\]
The appropriate subspace that will contain all $v^k$ functions was studied by Dafni \cite{Dafni} and Bourdaud \cite{bourdaud2002bmo}.
\begin{definition}\cite{Dafni}
$vmo(\mathbb{R}^n)$ is the subspace of $bmo(\mathbb{R}^n)$ functions satisfying
\begin{equation}
\label{vmo_cond_1}
\lim_{\delta\to 0} \sup\limits_{l(Q)\leq \delta} \frac{1}{|Q|} \int\limits_{Q} |f(x) - f_Q| dx =0
\end{equation}
and
\begin{equation}
\label{vmo_cond_2}
\lim\limits_{R \to \infty} \sup\limits_{l(Q)>1, Q\cap B(0,R)=\emptyset} \frac{1}{|Q|} \int\limits_{Q} |f(x)| dx =0.
\end{equation}
\end{definition}
\begin{theorem} [\cite{Dafni}] $vmo(\mathbb{R}^n)$ is the closure of $C_0(\mathbb{R}^n)$  in $bmo(\mathbb{R}^n)$.
\end{theorem}

An immediate consequence of this result and Theorem \ref{main_thm} is
\begin{theorem}
For $1\leq k\leq n$, the space $v^k(\mathbb{R}^n)$ is embedded into $vmo(\mathbb{R}^n)$.
\end{theorem}
\begin{cor}
$v^1(\mathbb{R}^n)$ does not contain $d^n(\mathbb{R}^n)$ as a subspace.
In particular, $v^k(\mathbb{R}^n)$ are proper subspaces of $d^k(\mathbb{R}^n)$ for $k=1,\dots,n$.
\end{cor}
\begin{proof}
Recall that $d^n(\mathbb{R}^n)$ coincides with $L^\infty(\mathbb{R}^n)$. If $L^\infty$ was a subspace of $v^1(\mathbb{R}^n)$, then by the last theorem we would have $L^\infty \subset  vmo(\mathbb{R}^n)$. However, choosing $f$ as a characteristic function of the quadrant $\{x=(x_1,\dots,x_n)\in \mathbb{R}^n: x_i>0\}$, we have an example of an $L^\infty$ function that does not satisfy (\ref{vmo_cond_1}). So $L^\infty \not\subset vmo(\mathbb{R}^n)$.
\end{proof}
Finally, we recall that $\mathcal{D}(\mathbb{R}^n)$ is dense in $W^{1,p}(\mathbb{R}^n)$ for any $p\in [1,\infty)$. Therefore by  Theorem \ref{thm_VS_my_version}, we have $W^{1,n} \subset v^{n-1}(\mathbb{R}^n)$.

All in all, we conclude that the following embeddings hold
\[
W^{1,n}(\mathbb{R}^n) \subset v^{n-1}(\mathbb{R}^n) \subset \dots \subset v^{1}(\mathbb{R}^n) \subset vmo(\mathbb{R}^n).
\]

\subsection{Intrinsic definition of the space $v^{n-1}$}
\begin{definition}
For $u\in d^{n-1}(\mathbb{R}^n) \cap C(\mathbb{R}^n)$, we will use the following notation
\[
\|u\|_*  = \sup\limits_{\partial \gamma=\emptyset}\frac{1}{|\gamma|} \left| \int_\gamma u(t) \tau(t) dt \right| + \sup\limits_{|\gamma|\geq 1}\frac{1}{|\gamma|} \left| \int_\gamma u(t) \tau(t) dt \right|,
\]
where the suprema are taken over smooth curves $\gamma$ with finite lengths $|\gamma|$, boundaries $\partial \gamma$ and  $\tau$ is the unit tangent vector to the curve $\gamma$.
\end{definition}

Our goal is to prove the following result.
\begin{theorem}
\label{thm_from_Smirnov}
There are constants $c_1,c_2>0$ such that for every $u\in d^{n-1}(\mathbb{R}^n)\cap C(\mathbb{R}^n)$,
\[
c_1 \|u\|_* \leq \|u\|_{d^{n-1}} \leq c_2 \|u\|_*.
\]
\end{theorem}
The proof is based on the following three lemmas
\begin{lem}
There exists $C>0$ such that for any $\gamma$ with $\partial \gamma=\emptyset$ or $|\gamma|\geq 1$,
\[
\frac{1}{|\gamma|}\left| \int_\gamma u(y) \tau(y) dy \right| \leq C \|u\|_{d^{n-1}}.
\]
\end{lem}
\begin{proof}
The proof is based on the argument of Bourgain and Brezis \cite{BB_New_estimates_Laplacian_Hodge_systems}.

Let $\eta\geq 0$ be a smooth radial function on $\mathbb{R}^n$, compactly supported in $|x|\leq 1$, such that $\|\eta\|_{L^1} = 1$.  As usual we put $\eta_\epsilon(x) = \epsilon^{-n} \eta(x/\epsilon)$. Let us define the $(n-1)$-form
\[
\Phi^\epsilon(x) = \sum_{j=1}^n \left( \int_\gamma \eta_\epsilon(t-x) \tau_j(t) dt\right) dx^{I_j}, x\in \mathbb{R}^n,
\]
where $I_j=(i_1,\dots,i_{n-1}), i_k\neq j$.

The reason to introduce this differential form is the following equality
\[
\left| \int_\gamma u(t) \tau(t)   dt\right| = \lim\limits_{\epsilon\to 0} \left| \int_\gamma \tau(t) \int_{\mathbb{R}^n} u(x) \eta_\epsilon(x-t) dx  dt\right| =
\]
\[
\lim\limits_{\epsilon\to 0} \left| \int u(x) \Phi^\epsilon(x) dx \right|.
\]
By the Remark \ref{rem_d_space_upsilon}, we need to estimate $\|\Phi^\epsilon\|_{\Upsilon^1_{n-1}}$.
It is clear that $\|\Phi^\epsilon\|_{L^1_{n-1}} \leq n \|\eta_\epsilon\|_{L^1} |\gamma|= n |\gamma|$.
Moreover,
\[
d\Phi^\epsilon(x) = -\left( \int_\gamma \nabla\eta_\epsilon(y-x)\cdot \tau(y) dy \right) dx^1 \wedge \dots \wedge dx^n =
\]
\[
= [\eta_\epsilon(a-x) - \eta_\epsilon(b-x)] dx^1 \wedge \dots \wedge dx^n.
\]
Therefore $\|d\Phi^\epsilon\|_{L^1_{n}}$ is $0$ if $\gamma$ is closed or $\leq 2$ if $\gamma$ is not closed.
Finally,
\[
\frac{1}{|\gamma|} \left| \int_\gamma u(s) \tau(s)   ds\right| \leq \frac{1}{|\gamma|} \limsup_{\epsilon\to 0}\left| \int u(x) \Phi^\epsilon dx \right| \leq \|u\|_{d^k} (2+n),
\]
because, for non-closed $\gamma$, $|\gamma|\geq 1$. So we proved the lemma with $C=n+2$.
\end{proof}

In order to prove the converse estimate, Bourgain and Brezis evoked the decomposition theorem of Smirnov.

\begin{theorem} [\cite{smirnov1993decomposition}]
\label{thm_Smirnov_B}
For any compactly supported $\Phi\in L^1_{n-1}(\mathbb{R}^n)$,  with $d\Phi=0$, there exists a sequence of positive numbers $\{\mu^m_j\}$ and closed smooth curves $\{\gamma^m_j\}$ such that for all $m\geq 1$,
\[
\sum\limits_{j=1}^\infty |\mu^m_j| |\gamma^m_j| \leq \|\Phi\|_{L^1_{n-1}}
\]
and for every $u\in C(\mathbb{R}^n)$ and $1\leq i\leq n$
\[
\sum\limits_{j=1}^\infty \mu^m_j \int_{\gamma^m_j} u(s) \tau_i(s) ds \to \int u(x) \phi_i(x) dx, \text{ as } m\to \infty,
\]
where $\phi_i$ are the components of $\Phi$.
\end{theorem}

In our case $d \Phi \in L^1_{n-1}(\mathbb{R}^n)$ does not necessarily vanish and we need a more general version of Smirnov's theorem, which we formulate in the following form
\begin{theorem} [\cite{smirnov1993decomposition}]
\label{thm_Smirnov_C}
Let $\Phi\in \Upsilon^1_{n-1}(\mathbb{R}^n)$. Then there exist $P\in \Upsilon^1_{n-1}(\mathbb{R}^n)$ and $Q\in \Upsilon^1_{n-1}(\mathbb{R}^n)$ such that
\begin{itemize}
    \item $\|\Phi\|_{L^1_{n-1}} = \|P\|_{L^1_{n-1}} + \|Q\|_{L^1_{n-1}}$,
    \item $dP=0$ and we can apply the previous theorem to $P$ \item $d Q = d\Phi$.
\end{itemize}
Moreover, there exist $\{\lambda^l_j\}$ and smooth curves  $\tilde{\gamma}^l_j$ (not necessarily closed) such that for all $l\geq 1$
\[
\sum_{j=1}^\infty |\lambda^l_j| |\tilde{\gamma}^l_j| \leq \|Q\|_{L^1_{n-1}},
\]
\[
\sum_{j=1}^\infty |\lambda^l_j| \leq \|dQ\|_{L^1_{n}}
\]
and for $1\leq i\leq n$
\[
\sum\limits_{j=1}^\infty \lambda^l_j \int_{\tilde{\gamma}^l_j} u(s) \tau_i(s) ds \to \int u(x) q_i(x) dx, \text{ as } l\to \infty.
\]
where $q_i$ are the components of $Q$.
\end{theorem}
Let us introduce an auxiliary norm for $u\in C(\mathbb{R}^n)$:
\[
\|u\|_{**} = \sup\limits_{\partial \gamma=\emptyset}\frac{1}{|\gamma|} \left| \int_\gamma u(s) \tau(s) ds \right| + \sup\limits_{|\gamma|< 1} \left| \int_\gamma u(s) \tau(s) ds \right|
\]
\[
+ \sup\limits_{|\gamma|\geq 1}\frac{1}{|\gamma|} \left| \int_\gamma u(s) \tau(s) ds \right|.
\]

\begin{lem}
For any $u\in d^{n-1}(\mathbb{R}^n)\cap C(\mathbb{R}^n)$,
\[
\|u\|_{d^{n-1}(\mathbb{R}^n)} \leq 2 \|u\|_{**}.
\]
\end{lem}
\begin{proof}
By the definition of $d^{n-1}(\mathbb{R}^n)$, there exists \[
\Phi = \sum\limits_{i=1}^n \phi_i dx^1\wedge \dots \widehat{dx^i} \wedge \dots dx^n  \in \mathcal{D}^{n-1}(\mathbb{R}^n)
\]
such that
\[
\|\Phi\|_{L^1_{n-1}} + \|d\Phi\|_{L^1} \leq 1
\]
and
\begin{equation}
\label{temp_fg_1}
\|u\|_{d^{n-1}} \leq 2 \max\limits_{I} |u(\phi_I)|.
\end{equation}

Let us apply Theorem \ref{thm_Smirnov_C} to $\Phi$. Then $\Phi$ can be decomposed into the sum of $P$ and $Q$ such that $d\Phi=dQ$, $\|\Phi\|_{L^1_{n-1}}  = \|P\|_{L^1_{n-1}}+\|Q\|_{L^1_{n-1}}  $ and $Q$ is a weak limit of the linear combination of the curves $\tilde{\gamma}^l_j$ in the sense that
\[
\sum\limits_{j=1}^\infty \tilde{\lambda}^l_j \int_{\tilde{\gamma}^l_j} u(s) \tau_i(s) ds \to \int u(x) q_i(x) dx, \text{ as } l\to \infty,
\]
where
\[
\sum_{j=1}^\infty |\tilde{\lambda}^l_j| (1+|\tilde{\gamma}^l_j|) \leq \|Q\|_{L^1_{n-1}} + \|dQ\|_{L^1} \leq 1, \text{ for all } l\geq 1.
\]
Moreover, applying Theorem \ref{thm_Smirnov_B} to $P$, we get a sequence of closed curves $\gamma^l_j$ and numbers $\lambda^l_j$ such that
\[
\sum\limits_{j=1}^\infty \lambda^l_j \int_{\gamma^l_j} u(s) \tau_i(s) ds \to \int u(x) p_i(x) dx, \text{ as } l\to \infty
\]
and
\[
\sum_{j=1}^\infty |\lambda^l_j| |\gamma^l_j| \leq \|P\|_{L^1_{n-1}}\leq 1 \text{ for all } l\geq 1.
\]

All in all,
\[
\int u(x) \phi_i(x) dx  = \lim\limits_{l\to \infty} \sum\limits_{j=1}^\infty \lambda^l_j \int_{\gamma^l_j} u(s) \tau_i(s) ds + \sum\limits_{j=1}^\infty \tilde{\lambda}^l_j \int_{\tilde{\gamma}^l_j} u(s) \tau_i(s) ds
\]
and
\begin{equation}
\label{temp_fg_2}
\left| \int u(x) \phi_i(x) dx \right| \leq \sup\limits_{l,j} \left| \frac{1}{|\gamma^l_j|} \int_{\gamma^l_j} u(s) \tau_i(s) ds \right|  +
\end{equation}
\[
+ \sup\limits_{l,|\tilde{\gamma}^l_j|< 1} \left| \int_{\tilde{\gamma}^l_j} u(s) \tau_i(s) ds \right| + \sup\limits_{l,|\tilde{\gamma}^l_j|\geq 1} \left| \frac{1}{|\tilde{\gamma}^l_j|} \int_{\tilde{\gamma}^l_j} u(s) \tau_i(s) ds \right| \leq \|u\|_{**}.
\]
The result follows from (\ref{temp_fg_1}) and (\ref{temp_fg_2}).
\end{proof}

\begin{lem}
For any $u\in C(\mathbb{R}^n)$
\[
\|u\|_{*} \leq \|u\|_{**} \leq 4 \|u\|_*.
\]
\end{lem}
\begin{proof}
The first inequality follows from the definitions of the norms. In order to see the second one, we need to show that
\[
\sup\limits_{|\gamma|< 1} \left| \int_\gamma u(s) \tau(s) ds \right| \leq  \sup\limits_{\partial \gamma=\emptyset}\frac{3}{|\gamma|} \left| \int_\gamma u(s) \tau(s) ds \right| + \sup\limits_{|\gamma|\geq 1}\frac{3}{|\gamma|} \left| \int_\gamma u(s) \tau(s) ds \right|.
\]
Let us consider any $\gamma$ with $|\gamma|<1$ and $\partial \gamma  = \{a,b\}$. We can always find $\gamma'$ such that $1<|\gamma'|<2$ and $\gamma'':= \gamma + \gamma'$ is a closed curve.

Then
\[
\left| \int_\gamma u(s) \tau(s) ds \right| \leq \left| \int_{\gamma''} u(s) \tau(s) ds \right| + \left| \int_{\gamma'} u(s) \tau(s) ds \right|
\]
\[
\leq \left| \frac{3}{|\gamma''|}\int_{\gamma''} u(s) \tau(s) ds \right| + \left| \frac{3}{|\gamma'|}\int_{\gamma'} u(s) \tau(s) ds \right|
\]
\end{proof}

\subsection{Examples of $d^k(\mathbb{R}^n)$ functions}
In this section, we want to show that there are more functions in $d^k(\mathbb{R}^n)$ besides those in $W^{1,n}(\mathbb{R}^n)$.

\subsubsection{Triebel-Lizorkin and Besov functions}

We recall that Sobolev space $W^{s,p}(\mathbb{R}^n)$, $1<p<\infty$ is a special case of more general classes of functions
\[
W^{s,p}(\mathbb{R}^n) = F^{s,p}_p(\mathbb{R}^n) = B^{s,p}_p(\mathbb{R}^n),
\]
here $F^{s,p}_q$, $s\in \mathbb{R}$, $0<p,q<\infty$ is the space of Triebel-Lizorkin and $B^{s,p}_q(\mathbb{R}^n)$, $s\in \mathbb{R}$, $0<p,q\leq \infty$, is the Besov space (see e.g. \cite{grafakos2009modern} or \cite{Triebel_Function_spaces_2} for definitions).

It was shown in \cite{van2010limiting} (see Proposition 2.1 there),  that
$\mathring{F}^{s,p}_q \subset D_{n-1}$ for all $sp=n$, $1<p<\infty$, $0<q<\infty$ (here $\mathring{F}^{s,p}_q$ is a homogeneous Triebel-Lizorkin space). Recalling the embedding theorems (see e.g. Ex 6.5.2 in \cite{grafakos2009modern})
\[
\mathring{B}^{s,p}_{\min(p,q)} \subset \mathring{F}^{s,p}_q \subset \mathring{B}^{s,p}_{\max(p,q)},
\]
and
\[
\mathring{B}^{s,p}_q \subset \mathring{B}^{s',p'}_q, \text{ if } sp=s'p' \text{ and } s>s'
\]
one can obtain the embedding $\mathring{B}^{s,p}_q\subset D_{n-1}$ for $0<q<\infty$. The case $q=\infty$ remains open (see Open problem 1 in \cite{Schaftingen2014}).

One can notice that the proof of Proposition 2.1 in \cite{van2010limiting} is exactly the same as the proof of Theorem 1.5 in \cite{van2004simple}. In fact it can be extended to the non-homogeneous setting as
\begin{theorem}
Let $1<p<\infty$, $1<q<\infty$. Then there exists constants $C_1$ and $C_2$ such that
\[
\|u\|_{d^{n-1}} \leq C_1 \|u\|_{F^{n/p,p}_q}
\]
and
\[
\|u\|_{d^{n-1}} \leq C_2 \|u\|_{B^{n/p,p}_q}.
\]
\end{theorem}

\subsubsection{Locally Lipschitz functions}
The following proposition provides a simple sufficient condition to ensure that $u\in d^{n-1}(\mathbb{R}^n)$.

\begin{prop}
Let $u\in W^{1,1}_{loc} (\mathbb{R}^n\setminus \{0\})$. If $|x|(u(x)+ \nabla u(x))\in L^\infty(\mathbb{R}^n)$, then $u\in d^{n-1}(\mathbb{R}^n)$ and
\[
\|u\|_{d^{n-1}} \leq C \||x|(|u|+ |\nabla u|)\|_{L^\infty}.
\]
\end{prop}
\begin{proof}
The proof follows from  integration by parts as in the proof of Proposition 4.3 in \cite{Schaftingen_function_spaces_BMO_Sobolev}.

We need to show that for any $\Phi = \sum\limits_{j=1}^n \phi_j(x) dx^1\wedge \dots \widehat{dx}^j \wedge \dots dx^n \in \mathcal{D}^{n-1}(\mathbb{R}^n)$, we have
\[
\left|\int u(x)\phi_j(x) dx \right| \leq C \||x|(u(x)+ \nabla u(x))\|_{L^\infty} \|\Phi\|_{\Upsilon^1_{n-1}}.
\]

Note that
\[
 \int x_j (\sum_{i} \phi_i \partial_i u) dx  =  - \int \phi_j u dx  - \int x_j u \cdot (\sum_{i} \partial_i \phi_i)  dx.
\]
So
\[
\left| \int u(x)\phi_j(x) dx \right| \leq n \||x| \nabla u\|_{L^\infty} \|\Phi\|_{L^1_{n-1}} + \||x| u\|_{L^\infty} \|d\Phi\|_{L^1_{n}}.
\]
\end{proof}

The proposition allows us to give an example of $u\in d^{n-1}$ which is not covered by the previous classes of functions, the Bessel potential $G_n$.

\begin{remark}
\label{rem_3.5.3}
A typical example of $u\in D^{n-1}\setminus W^{1,n}$ in \cite{Schaftingen_function_spaces_BMO_Sobolev} is the function $u(x) = \log |x|$. However, this function does not belong to $bmo(\mathbb{R}^n)$ and therefore is not in any $d^k$, $1\leq k \leq n$ as
\[
\sup_{|Q|>1} \frac{1}{|Q|}\int_{Q} |\log|y|| dy  = \infty.
\]
\end{remark}

\begin{example}
\label{ex_last}
Let $G_n(x)$ be the Bessel potential of order $n$, i.e. the function whose Fourier transforms is given by $\hat{G}_n(\xi) = (1+|\xi|^2)^{-n/2}$.

The fact that $G_n$ satisfies the conditions of the last proposition follows from the fact that $G_n$ is a continuously differentiable function on $\mathbb{R}^n\setminus \{0\}$ and the asymptotic formulas  for the Bessel potentials (see e.g. \cite{aronszajn1961theory}, pp. 415-417):
\[
G_n(x) \sim C_1 \log|x|, \text{ as } x\to 0,
\]
\[
G_n(x) \sim C_2 |x|^{-1/2} e^{-|x|}, \text{ as } x\to \infty.
\]
Moreover,
\[
\frac{\partial}{\partial x_i} G_n(x) = C'_s \cdot \frac{x_i}{|x|}  K_{1}(|x|),
\]
where $K_1$ is the Bessel-Macdonald function of order $1$, with the asymptotics
\[
K_1(r) \sim C_3 r^{-1} , \text{ as } r\to 0+
\]
\[
K_1(r) \sim C_4 r^{-1/2} e^{-r} , \text{ as } r\to \infty.
\]
\end{example}

\subsection{Application to PDE}

We will illustrate how non-homogeneous $d^k$ spaces can be used in the analysis of classic PDE.

The following result was shown in \cite{bourgain2007new} (see Theorems 2 and 3 there): if $\Delta U = F$ in $\mathbb{R}^n$ and $div F=0$, then
\[
\|U\|_\infty + \|\nabla U\|_2 \leq C \|F\|_1, \text{ if } n=2
\]
and
\[
\|U\|_{n/(n-2)} + \|\nabla U\|_{n/(n-1)} \leq C \|F\|_1, \text{ if } n\geq 3.
\]

A more general result of Bourgain and Brezis  (see Theorem $4'$ in \cite{bourgain2007new} and Remark 2.1 in \cite{Brezis_Schaftingen}) implies that one can relax the condition $div F=0$ to $div F \in L^1$ to obtain
\[
\|\nabla U\|_{n/(n-1)} \leq C (\|F\|_1 + \|div F\|_1).
\]
Note that for $n\geq 3$ this can be combined with a Sobolev embedding theorem to produce
\[
\|U\|_{n/(n-2)} + \|\nabla U\|_{n/(n-1)} \leq C (\|F\|_1 + \|div F\|_1).
\]
However (as noted in \cite{Brezis_Schaftingen}), if $n=2$ then  $U$ may no longer be an $L^\infty$ vector field.
Let us explain why it may happen using Theorem \ref{main_thm}.

Let $g(x)=\log|x|$. Then $g*F$ is continuous for any $F\in \Upsilon^1_{1}$ and if
\[
\|U\|_\infty = (2\pi)^{-1} \|g*F\|_\infty \leq C (\|F\|_1 + \|div F\|_1)
\]
were true for any $F\in \mathcal{D}^1(\mathbb{R}^2)$, then we would have
\[
|g*F(0)| = |\int g(x) F(x) dx | \leq C \|F\|_{\Upsilon^{1}_{1}},
\]
and $g(x)=\log|x|$ would be an $d^{1}$ function and by Theorem \ref{main_thm}, $\log|x|\in bmo(\mathbb{R}^2)$. However, this is false by Remark \ref{rem_3.5.3}.

So the solution of equation $\Delta U = F\in \mathbb{R}^2$ can be essentially unbounded even if $div F\in L^1$, because the fundamental solution of $\Delta$ in $\mathbb{R}^2$ is not an element of $d^{1}(\mathbb{R}^2)$.
Based on the examples of $d^{n-1}(\mathbb{R}^n)$ functions, one can guess that the situation should be better in the case of the Helmholtz equation.
Indeed, the following proposition shows that solutions to the Helmholtz equation can be fully controlled even under relaxed conditions.
\begin{theorem}
\label{prop_application}
 Let $n\geq 2$, $F\in L^1(\mathbb{R}^n;\mathbb{R}^n)$ and $div F\in L^1(\mathbb{R}^n)$. Then the system $(I-\Delta) U = F$ admits a unique solution $U$ such that
\begin{itemize}
    \item If $n=2$, then
    \[
    \|U\|_{\infty} + \|\nabla U\|_{2} \leq C (\|F\|_1 + \|div F\|_1)
    \]
    \item If $n\geq 3$, then
    \[
    \|U\|_{n/(n-2)} + \|\nabla U\|_{n/(n-1)} \leq C (\|F\|_1 + \|div F\|_1)
    \]
\end{itemize}
\end{theorem}
\begin{proof}
Without loss of generality we can assume that $F\in \mathcal{S}(\mathbb{R}^2;\mathbb{R}^2)$.

\textit{Case 1:} If $n\geq 3$, then
\[
\Delta U = \frac{\Delta}{I-\Delta} F =: \tilde{F}.
\]
As $\frac{\Delta}{I-\Delta}$ is an operator of convolution against a finite measure (see e.g. Chapter 5 in \cite{stein2016singular}), $\tilde{F}\in L^1$ and $div \tilde{F} = \frac{\Delta}{I-\Delta} div F \in L^1$, with
\[
\|\tilde{F}\|_1 + \|div \tilde{F}\|_1 \leq C( \|F\|_1 + \|div F\|_1).
\]
Hence by Theorem $4'$ in \cite{bourgain2007new},
\[
\|\nabla U\|_{n/(n-1)} \leq C (\|\tilde{F}\|_1 + \|div \tilde{F}\|_1) \leq C'( \|F\|_1 + \|div F\|_1)
\]
and the application of Sobolev's embedding theorem completes the proof.

\textit{Case 2:} If $n=2$, then solution $U$ has the form $U(x) = G_2*F(x)$, where $G_2(x)$ is the Bessel potential of order $2$. By  Example \ref{ex_last}, $G_2\in d^{1}(\mathbb{R}^2)$.
Thus for any $x\in \mathbb{R}^2$,
\[
|U(x)| = |G_2*F(x)|  \leq \|G_2\|_{d^1} \|\tau_x F\|_{\Upsilon^1_1(\mathbb{R}^2)} = \|G_2\|_{d^1} \| F\|_{\Upsilon^1_1(\mathbb{R}^2)},
\]
where $\tau_x$ is the translation operator defined by $(\tau_x f)(y) = f(y-x)$.
In other words
\begin{equation}
    \label{appl_temp1}
    \|U\|_\infty \leq C (\|F\|_1 + \|div F\|_1 ).
\end{equation}
In order to control $\nabla U$ notice that the decay of $F$ and $G_2$ implies
\[
\int |\nabla U_i(x)|^2 dx =  - \int U_i(x) \Delta U_i(x) dx  =
\]
\[
=  \int U_i(x) F_i(x) dx   - \int U^2_i(x) dx.
\]
Hence, recalling that $U$ is a convolution of the $L^1$ functions $G_2$ and $F$,
\[
\|\nabla U\|_2 \leq C \|U\|^{1/2}_\infty ( \|F\|_1 + \|U\|_1)^{1/2} \leq C \|U\|^{1/2}_\infty \|F\|^{1/2}_1.
\]
 Using (\ref{appl_temp1}) we complete the proof.
\end{proof}

\section{$d^k$ spaces on Lipschitz domains}
In this section we define $d^k$ classes on domains. Everywhere in this section we assume $\Omega$ to be a bounded Lipschitz domain in $\mathbb{R}^n$.

\begin{definition}
Let $1\leq k\leq n$. A distribution $u\in \mathcal{D}'(\Omega)$ is said to belong to $d^k(\Omega)$ if there exists $C>0$ such that $|u(\phi_I)| \leq C \|\Phi\|_{\Upsilon^1_k(\Omega)}$ for any
\[
\Phi= \sum\limits_{|I|=k} \phi_I dx^I \in \mathcal{D}^k(\Omega).
\]
\end{definition}
We denote the space of such distributions by $d^k(\Omega)$ and equip it with the norm
\[
\|u\|_{d^k(\Omega)} := \sup\{ |u(\phi_I)| : \Phi\in \mathcal{D}^k(\Omega); \|\Phi\|_{\Upsilon^1_k(\Omega)}\leq 1 \}.
\]
\begin{remark}
Let $1\leq k\leq n$. We want to consider distributions $u\in \mathcal{E}'(\Omega)$ such that $|u(\phi_I)| \leq C \|\Phi\|_{\Upsilon^1_k(\Omega)}$ for some finite $C>0$ and any
\[
\Phi= \sum\limits_{|I|=k} \phi_I dx^I \in \mathcal{D}^k(\mathbb{R}^n).
\]
The class of $\mathcal{E}'(\Omega)\cap d^k(\mathbb{R}^n)$, equipped with the norm $\|\cdot\|_{d^k(\mathbb{R}^n)}$ forms an incomplete normed space. Therefore we define  $d^k_z(\Omega)$ as follows.
\end{remark}

\begin{remark}
\label{rem_d_local_def}
The definitions we use were suggested by Van Schaftingen in \cite{Schaftingen_function_spaces_BMO_Sobolev}. It is also possible to define $d^k(\Omega)$  as we did in Remark \ref{rem_d_space_upsilon}. Any $u\in \mathcal{D}'(\Omega)$ defines a linear map $\tilde{u}:\mathcal{D}^k(\Omega) \to \mathbb{R}^{\binom{n}{k}}$ by
\[
\tilde{u}\left(\sum\limits_{|I|=k} \phi_I dx^I \right) = (u(\phi_I))_{I}
\]
and $u\in d^k(\Omega)$ if and only if $\tilde{u}$ can be extended to a bounded linear map from $\Upsilon^1_{k,0}(\Omega)$ to $(\mathbb{R}^{\binom{n}{k}},\|\cdot\|_{max})$, where
$\Upsilon^1_{k,0}(\Omega) = \overline{\mathcal{D}^k(\Omega)}$ and the closure is taken with respect to the $\Upsilon^1_k$ norm.
\end{remark}

\subsection{$d^k_z(\Omega)$ spaces}

All properties of $d^k_z(\Omega)$ spaces can be deduced from the previous results and the following definition
\begin{definition}
\label{prop_d_z}
Let $1\leq k\leq n$. Then
\[
d^k_z(\Omega) = \{u \in d^k(\mathbb{R}^n): supp \  u \in \overline{\Omega}\}.
\]
\end{definition}
\begin{remark}
It is clear that $d^k_z(\Omega)$ is a closed subspace of $d^k(\mathbb{R}^n)$, hence complete, and  $\mathcal{E}'(\Omega)\cap d^k(\mathbb{R}^n)\subset d^k_z(\Omega) $.
Conversely, any $u\in d^k_z$ is the weak limit of $\mathcal{E}'(\Omega)\cap d^k(\mathbb{R}^n)$. Indeed, consider any $u\in d^k(\mathbb{R}^n)$ supported in $\bar{\Omega}$.
By Theorem \ref{main_thm} and the definition of $bmo_z(\bar{\Omega})$, $u\in bmo_z(\Omega)$. In particular $u\in L^1(\Omega)$. Let $\eta_j$ be a sequence of $\mathcal{D}(\Omega)$ functions such that $\lim\limits_{j\to \infty} \eta_j  = \chi_{\Omega}$, the characteristic function of $\Omega$. Then by Lebesgue's dominated convergence theorem, for any $\Phi\in \mathcal{D}^k(\bar{\Omega})$ and $I$,
\[
\int_{\Omega} u(x) \phi_I(x) dx = \lim_{j\to \infty} \int_{\Omega} (\eta_j u)(x) \phi_I(x) dx.
\]
This shows that $u = \lim\limits_{j\to \infty} (\eta_j u)$ is a weak limit.
\end{remark}
Combining this definition with Lemma \ref{lem_emb} we obtain
\begin{prop}
The spaces $d^k_z(\Omega)$ form a monotone family, i.e.\ the following embeddings hold
\[
d^n_z(\Omega) \subset d^{n-1}_z(\Omega) \subset \dots \subset d^1_z(\Omega).
\]
\end{prop}

\begin{prop}
Let $\Omega$ be a bounded Lipschitz domain and $W^{1,n}_0(\Omega)$ be the closure of $\mathcal{D}(\Omega)$ functions in the norm $\|\cdot\|_{W^{1,n}(\Omega)}$. Then $W^{1,n}_0(\Omega)$ is continuously embedded into $d^{n-1}_z(\Omega)$.
\end{prop}
\begin{proof}
The space $W^{1,n}_0(\Omega)$ can be characterized (see e.g. Theorem 5.29 in \cite{adams2003sobolev}) as follows: let $f\in W^{1,n}(\Omega)$, then $f\in W^{1,n}_0(\Omega)$ if and only if the extension of $f$ by zero to $\mathbb{R}^n\setminus \bar{\Omega}$ belongs to $W^{1,n}(\mathbb{R}^n)$. Using this characterization, we can identify any $u\in W^{1,n}_0(\Omega)$ with $\tilde{u}\in W^{1,n}(\mathbb{R}^n)$ supported in $\bar{\Omega}$. By Van Schaftingen's theorem such $\tilde{u}$ is an element of $ d^{n-1}(\mathbb{R}^n)$ and is supported in $\bar{\Omega}$. Therefore by Defintion \ref{prop_d_z},  $\tilde{u}\in d^{n-1}_z(\Omega)$.
\end{proof}

\begin{prop}
The space $d^1_z(\Omega)$ is a proper subspace of $bmo_z(\Omega)$.
\end{prop}
\begin{proof}
It follows immediately from Theorem \ref{main_thm}, Defintion \ref{prop_d_z} and the definition of $bmo_z(\Omega)$.
\end{proof}
All in all, we can see that the spaces $d^k_z(\Omega)$ form a family of intermediate spaces between $W^{1,n}_0(\Omega)$ and $bmo_z(\Omega)$. 

\subsection{$d^k(\Omega)$ spaces}
It follows directly from the definitions of $d^k(\mathbb{R}^n)$ and $d^k(\Omega)$, that $u\to u|_{\Omega}$ maps $d^k(\mathbb{R}^n)$ to $d^k(\Omega)$ and
\begin{equation}
\label{d_r_and_d}
\|u|_{\Omega}\|_{d^k(\Omega)} \leq \|u\|_{d^k(\mathbb{R}^n)},
\end{equation}
where $u|_{\Omega}$ stands for the restriction of $u$ to $\Omega$.

Repeating verbatim the proof of Proposition \ref{lem_emb}, one obtains
\begin{prop}
Let $1\leq k<l\leq n$ and $u\in d^l(\Omega)$. Then $u\in d^k(\Omega)$ and $\|u\|_{d^k(\Omega)} \leq \|u\|_{d^l(\Omega)}$. In other words
\[
d^n (\Omega) \subset d^{n-1} (\Omega) \subset \cdots \subset d^{1} (\Omega).
\]
\end{prop}

In order to show that $W^{1,n}(\Omega)\subset d^{n-1}(\Omega)$, we recall the extension property of Sobolev spaces. It is well-known (see e.g. Theorem 5.24 in \cite{adams2003sobolev}) that if $\Omega$ is a Lipschitz domain then there exists a bounded linear operator $E:W^{l,p}(\Omega)\to W^{l,p}(\mathbb{R}^n)$ such that  $Eu = u$ almost everywhere in $\Omega$ for all $u\in W^{l,p}(\Omega)$. If we consider such an extension $E$ on $W^{1,n}(\Omega)$ and recall (\ref{d_r_and_d}) and Theorem \ref{thm_VS_my_version}, then
\[
\|u\|_{d^{n-1}(\Omega)} = \|Eu|_{\Omega}\|_{d^{n-1}(\Omega)}\leq \|Eu\|_{d^{n-1}(\mathbb{R}^n)} \leq
\]
\[
\leq \|Eu\|_{W^{1,n}(\mathbb{R}^n)} \leq \|E\| \|u\|_{W^{1,n}(\Omega)}.
\]
In other words,
\begin{prop}
If $\Omega$ is a bounded Lipschitz domain, then $W^{1,n}(\Omega)$ is continuously embedded into $d^{n-1}(\Omega)$.
\end{prop}

The following result is the analogue of Theorem \ref{main_thm} on Lipschitz domains.
\begin{theorem}
\label{thm_3}
Any $u\in d^1(\Omega)$ is a $bmo_r(\Omega)$ function and
\[
\|u\|_{bmo_r(\Omega)} \leq C \|u\|_{d^1(\Omega)}.
\]
\end{theorem}
The proof is more technical than the one of Theorem \ref{main_thm} because of the presence of $\partial \Omega$. Firstly, we state a corollary of the Ne\v cas inequality:
\[
\|f\|_{L^2(\Omega)} \leq C( \|f\|_{W^{-1,2}(\Omega)} + \|\nabla f\|_{W^{-1,2}(\Omega)}) \forall f\in L^2(\Omega).
\]

\begin{lem} [\cite{auscher2005hardy}, Lemma 10]
\label{lem_auscher}
Let $\Omega$ be a bounded Lipschitz domain in $\mathbb{R}^n$. If $g\in L^2(\Omega)$ and $\int g  = 0$, then there exists a vector-valued function $F\in W^{1,2}_0(\Omega, \mathbb{R}^n)$ such that
\[
\left\{
\begin{array}{l}
    \rm{div} F = g, \text{ in } \Omega   \\
    \|DF\|_{L^2} \leq C\|g\|_2.
\end{array}
\right.
\]
Here $DF$ is a matrix $\partial_j F_i$ and $C>0$ depends only on the Lipschitz constant of $\Omega$.
\end{lem}

Using this lemma we prove the following
\begin{lem}
\label{lem_temp2}
Let $\Omega$ be a bounded Lipschitz domain in $\mathbb{R}^n$. If $g\in H^1_z(\Omega)$, then there exists a vector-valued function $F\in W^{1,1}_0(\Omega, \mathbb{R}^n)$ such that
\[
\left\{
\begin{array}{l}
    \rm{div} F = g, \text{ in } \Omega   \\
    \|DF\|_{L^1} \leq C\|g\|_{H^1}.
\end{array}
\right.
\]
\end{lem}
\begin{proof}
Let $g\in H^1_z(\Omega)$. Then by Theorem \ref{thm_CKS_H1_atoms}, it can be decomposed into $H^1_z(\Omega)$ atoms $a_i\in L^2(\mathbb{R}^n)$  as
\[
g = \sum\limits_{i=1}^\infty \lambda_i a_i
\]
and
\[
\sum\limits_{i=1}^\infty |\lambda_i| \leq 2 \|g\|_{H^1}.
\]
For each $i\geq 1$, by means of Lemma \ref{lem_auscher}, we can find  $V^i\in W^{1,2}_0(Q_i, \mathbb{R}^n)$, such that
\[
\left\{
\begin{array}{l}
    \rm{div} V^i = a_i  \text{ in } Q_i   \\
    \|DV^i\|_{L^2} \leq C\|a_i\|_{L^2}.
\end{array}
\right.
\]

As $W^{1,2}_0(Q_i)$ fields, $V^i$ can be continuously extended by 0 to $W^{1,2}(\Omega)$. We denote these extensions by the same $V^i$. We claim that $F  = \sum\limits_{i=1}^\infty \lambda_i V^i$ is the solution we seek.

Indeed, since $a_i$ are atoms, we have
\[
\|D V^i\|_{L^1} \leq |Q_i|^{1/2} \|D V^i\|_{L^2} \leq C |Q_i|^{1/2} \|a_i\|_{L^2} \leq C_1 \text{ for all } i\geq 1.
\]
Therefore, the partial sums $\sum_{i=1}^N \lambda_i V^i$, supported in $\Omega$, converge to an element $F$ of $W^{1,1}_0(\Omega, \mathbb{R}^{n\times n})$ and
\[
\|DF\|_{L^1} \leq C_1 \sum_i |\lambda_i| \leq  C\|g\|_{H^1}.
\]
Finally, by the construction of $F$,
\[
\textrm{div} F = \sum_i \lambda_i \cdot \textrm{div} V^i  = \sum_i \lambda_i a_i = g.
\]
\end{proof}

Now we can prove the last theorem of this section

\begin{proof} [Proof of Theorem \ref{thm_3}]
We will use the duality between $h^1_z(\Omega)$ and $bmo_r(\Omega)$ asserted by Theorem \ref{thm_bmor_as_dual}.
By Lemma \ref{lem_density_h1_domains}, it is enough to show that for any $f\in  \mathcal{D}(\Omega)$ and $u\in d^1(\Omega)$
\begin{equation}
\label{sq_0}
|u(f)|\leq C \|u\|_{d^1} \|f\|_{h^1}.
\end{equation}
Given $f\in \mathcal{D}(\Omega)$, we write $f$ as the sum $f=g+\theta$, where
\[
g = f - \int f(x) dx\cdot  \psi,
\]
\[
\theta = \int f(x) dx \cdot \psi,
\]
where $\psi\in \mathcal{D}(\Omega) $ is any function with $\int \psi = 1$.

Note that $\theta\in \mathcal{D}(\Omega)$ with $\|\theta\|_{h^1} \leq \|\psi\|_{L^1} \|f\|_{h^1}$ and  $\|\theta\|_{W^{1,1}} \leq \|f\|_{h^1} \|\psi\|_{W^{1,1}}$. Moreover if we define  $\Theta = \sum\limits_{i=1}^n \theta dx^i \in \mathcal{D}^1(\Omega)$, then $\|\Theta\|_{\Upsilon^1_{1}(\Omega)} \leq C \|\psi\|_{W^{1,1}} \|f\|_{h^1}$. Therefore
\begin{equation}
\label{sq_1}
|u(\theta)| \leq \|u\|_{d^1(\Omega)} \|\Theta\|_{\Upsilon^1_{1}(\Omega)} \leq C_\psi \|u\|_{d^1(\Omega)} \|f\|_{h^1}.
\end{equation}

On the other hand, for $g\in \mathcal{D}(\Omega)$, we recall Lemma \ref{lem_H1_h1} to see that $g\in H^1_z(\Omega)$ and
\begin{equation}
\label{sd}
\|g\|_{H^1} \leq C_{\Omega} \|g\|_{h^1}\leq C'_\psi \|f\|_{h^1}.
\end{equation}
Hence,  Lemma \ref{lem_temp2} is applicable and there exists $F\in W^{1,1}_0(\Omega; \mathbb{R}^n)$ such that
\[
\left\{
\begin{array}{l}
    \rm{div} F = g, \text{ in } \Omega   \\
    \|D F\|_{L^1(\Omega; \mathbb{R}^{n\times n})} \leq C \|g\|_{H^1}.
\end{array}
\right.
\]
Using this $F$, we introduce $n$ differential forms
\[
\Phi^j = \sum\limits_{i=1}^n \partial_i F_j dx^i
\]
and claim that all $\Phi^j\in \Upsilon^1_{1,0}(\Omega)$ and $\|\Phi^j\|_{\Upsilon^1_{1}(\Omega)} \leq C'_\psi \|f\|_{h^1}$ (recall that $\Upsilon^1_{k,0}(\Omega) = \overline{\mathcal{D}^k(\Omega)}$ where the closure is taken with respect to the $\Upsilon^1_k$ norm). Assuming the claim and recalling that $u$ is well defined on components of $\Upsilon^1_{1,0}(\Omega)$ forms (see Remark \ref{rem_d_local_def}), one has
\begin{equation}
\label{sq_2}
|u(g)| = |u(\sum\limits_{i=1}^n \partial_i F_i) | \leq \sum\limits_{i,j=1}^n |u(\partial_i F_j)| \leq
\end{equation}
\[
\leq n \|u\|_{d^1(\Omega)}  \max_{1\leq j \leq n} \|\Phi^j\|_{\Upsilon^1_1(\Omega)} \leq C \|u\|_{d^1(\Omega)}  \|f\|_{h^1}.
\]
We complete the proof by deducing (\ref{sq_0}) from (\ref{sq_1}), (\ref{sq_2}) and the triangle inequality.

In order to prove the claim, we note that $d\Phi^j=0$ by construction and all components of $\Phi^j$ are $L^1(\Omega)$ functions, bounded in the $L^1$-norm by a multiple of $\|g\|_{H^1}$. Recalling (\ref{sd}), we may conclude that
\[
\|\Phi^j\|_{\Upsilon^1_1(\Omega)} = \|\Phi^j\|_{L^1_1(\Omega)} \leq C \|f\|_{h^1}.
\]
Furthermore,  $F_j\in W^{1,1}_0(\Omega)$ for $j=1,\dots, n$, which means that there exist sequences $\{F^m_j\}_{m=1}^\infty \subset \mathcal{D}(\Omega)$ such that $\|\partial_i F^m_j - \partial_i F_j\|_{L^1(\Omega)} \to 0$, as $m\to \infty$. Hence, by forming  closed $\mathcal{D}^1(\Omega)$-forms
\[
\Phi^{j,m} = \sum\limits_{i=1}^n \partial_i F^m_j dx^i,
\]
we can construct $\mathcal{D}^1(\Omega)$ approximations of $\Phi^j$, such that as  $m\to \infty$,
\[
\|\Phi^{j,m} - \Phi^{j}\|_{\Upsilon^1_1(\Omega)}  = \|\Phi^{j,m} - \Phi^{j}\|_{L^1_1(\Omega)} \to 0,
\]
which shows that $\Phi^j\in \Upsilon^1_{1,0}(\Omega)$ for $j=1,\dots, n$.
\end{proof}

\section{$d^k$ spaces on Riemannian manifolds}
Let $(M,g)$ be a complete Riemannian manifold. Then $\exp_p$ is defined on $T_p M$ and, as mentioned earlier, for sufficiently small $r_p>0$, maps $B_{r_p}(0)\in T_p M$ diffeomorphically onto an open subset of $M$.
Let us denote by $\textrm{inj}_M(p)$, the supremum of all such $r_p>0$ and define the injectivity radius of $M$ as
\[
\textrm{inj}_M := \inf\{\textrm{inj}_M(p): p\in M\}.
\]
\begin{definition}
A Riemannian manifold $(M,g)$ is called a manifold with bounded geometry if
\begin{enumerate}
    \item $M$ is complete and connected;
    \item $\textrm{inj}_M >0$;
    \item For every multi-index $\alpha$, there exists $C_\alpha>0$ such that $|D^\alpha g_{i,j}| \leq C_\alpha$ in the normal geodesic coordinates $(\Omega_p(r_p), \exp^{-1}_p)$.
\end{enumerate}
\end{definition}

Examples of manifolds with  bounded geometry include compact  Riemannian manifold, $\mathbb{R}^n$ and $\mathbb{H}^n$ (see e.g. \cite{eldering2013normally}).

\subsection{Tame partition of unity}
Let $(M,g)$ be a Riemannian manifold with bounded geometry. For $\delta\in (0,\textrm{inj}_M)$, we denote by $\Omega_\delta(p)$, the image $B_\delta(0)$ by the map $\exp_p$ which is called a geodesic ball with radius $\delta$ centered at $p$.
\begin{prop}[\cite{Triebel_Function_spaces_2} p. 284]
 For sufficiently small $\delta>0$ there exists a uniformly locally finite covering of $M$ by a sequence of geodesic balls $\{\Omega_\delta(p_j)\}_{j\in \mathbb{Z}_+}$ and a corresponding smooth partition of unity $\{\psi_j\}_{j\in \mathbb{Z}_+}$ subordinate to $\{\Omega_\delta(p_j)\}_{j\in \mathbb{Z}_+}$.
\end{prop}
Such covering and partition of unity we will call following Taylor \cite{taylor2009hardy}, a tame covering and a tame partition of unity.

\subsection{$W^{s,p}(M)$, $h^1(M)$ and $\textrm{bmo}(M)$}
\begin{definition} [\cite{Triebel_Function_spaces_2}, Chapter 7]
Let $(M,g)$ be a Riemannian manifold with bounded geometry and let $\{\psi_j\}$ be a tame partition of unity subordinate to a tame cover by geodesic balls $\{\Omega_\delta(p_j)\}$. The Sobolev space $W^{s,p}(M)$, $1<p<\infty$, $s>0$ is defined as \[
W^{s,p}(M) = \{f\in \mathcal{D}'(M):  \sum\limits_{j\in \mathbb{Z}_+} \|\psi_j f\circ \exp_{p_j}\|^p_{W^{s,p}(\mathbb{R}^n)}<\infty \}
\]
\end{definition}

Taylor in \cite{taylor2009hardy}, introduced versions of Hardy spaces and $bmo$ on manifolds with bounded geometry. One way to define $h^1(M)$ is as follows:
\begin{definition} [\cite{taylor2009hardy} Corollary 2.4]
Let $f\in \mathcal{D}'(M)$ and $\{\psi_j\}$ a tame partition of unity subordinate to a tame cover by geodesic balls $\{\Omega_\delta(p_j)\}$. We say that $f\in h^1(M)$ if $
\sum_j \|(\psi_j f) \circ \exp_{p_j}\|_{h^1(\mathbb{R}^n)} < \infty$.
We equip the space $h^1(M)$ with the norm
\[
\|f\|_{h^1(M)} = \sum_j \|(\psi_j f) \circ \exp_{p_j}\|_{h^1(\mathbb{R}^n)}.
\]
\end{definition}

The space $bmo(M)$ is defined similarly
\begin{definition}[\cite{taylor2009hardy} Corollary 3.4]
Let $f\in L^1_{loc}(M)$ and $\{\psi_j\}$ a tame partition of unity subordinate to a tame cover by geodesic balls $\{\Omega_\delta(p_j)\}$. We say that $f\in bmo(M)$ if $
\sum_j \|(\psi_j f) \circ \exp_{p_j}\|_{bmo(\mathbb{R}^n)} < \infty$.
We equip the space $bmo(M)$ with the norm
\[
\|f\|_{bmo(M)} = \sum_j \|(\psi_j f) \circ \exp_{p_j}\|_{bmo(\mathbb{R}^n)}.
\]
\end{definition}

\begin{remark}
All these classes of functions have equivalent global definitions. However, for our purposes it is more convenient to use the introduced versions. We refer to \cite{taylor2009hardy}, \cite{aubin1982nonlinear} and \cite{Triebel_Function_spaces_2} for alternative definitions and the proofs of their equivalence.
\end{remark}

\subsection{$d^k(M)$ spaces and the embedding into $bmo(M)$}

\begin{definition}
Let $\{\psi_j\}$ be a tame partition of unity subordinate to a tame cover by geodesic balls $\{\Omega_\delta(p_j)\}$. We say that $u\in \mathcal{D}'(M)\in d^k(M)$ if for each $j$, $(\psi_j u) \circ \exp_{p_j} \in d^k(\mathbb{R}^n)$ and
\[
\|u\|_{d^k(M)} := \sum_j \|(\psi_j u) \circ \exp_{p_j}\|_{d^k(\mathbb{R}^n)} < \infty.
\]
\end{definition}

We complete this part with the result which immediately follows from the definitions of the spaces $W^{1,n}(M)$, $d^k(M)$, $bmo(M)$ and the results of Section 3.2: Lemma 3.2.4 and Theorems 3.2.5, 3.2.6.

\begin{theorem}
\label{thm_on_manifolds}
Let $M$ be the Riemannian manifold with bounded geometry. Then the following continuous embeddings are true
\[
W^{1,n}(M) \subset d^{n-1}(M) \subset \dots \subset d^1(M) \subset bmo(M)
\]
\end{theorem}

\section{Acknowledgement}
The author is very grateful to Galia Dafni for careful reading of the manuscript and many useful discussions concerning it. He would also like to show his gratitude to the referee for valuable suggestions and comments that considerably helped to improve the quality of this paper.

\bibliography{references}
\bibliographystyle{amsplain}

\address{Department of Mathematics and Statistics, Concordia University, Montreal, QC H3G 1M8, Canada}\\
\email{almaz.butaev@concordia.ca}

\end{document}